\documentclass[11pt,3p]{elsarticle}

\usepackage{graphicx}
\usepackage{algorithm}  
\usepackage{algorithmic}
\usepackage{amssymb,url,amsmath,amsthm}
\usepackage{bm} 
\usepackage[dvipsnames]{xcolor}
\usepackage{booktabs}	
\usepackage{epstopdf}
\usepackage{color}
\usepackage{pifont}
\usepackage{tikz}
\usepackage{hyperref}
\usepackage{comment}
\usepackage{natbib}
\usepackage{eurosym}
\biboptions{numbers,square,sort}

\usepackage{pgfplots}
\usepackage{setspace}
\usepackage{rotating}
\usepackage{circuitikz}
\usepackage{bbm}

\newtheorem{theorem}{Theorem}
\newtheorem{lemma}{Lemma}
\newtheorem{proposition}{Proposition}
\newtheorem{corollary}{Corollary}
          
\newtheorem{definition}{Definition}

\tikzset{>=latex}

\usepackage{bm}

\newcommand{\BC}{\mathcal{B}}
\newcommand{\VC}{\mathcal{V}}

\newcommand{\X}{\mathbb{X}}

\newcommand{\R}{\mathbb{R}}
\newcommand{\SB}{\mathbb{S}}
\journal{arXiv}

\allowdisplaybreaks
\allowdisplaybreaks

\begin{document}
\begin{frontmatter}

\title{A market-based approach for enabling inter-area reserve exchange}

\author[a1]{Orcun Karaca\corref{cor1}}
\ead{okaraca@ethz.ch}
\author[a3]{Stefanos Delikaraoglou}
\author[a2]{Maryam Kamgarpour}
\cortext[cor1]{Corresponding author}
\address[a1]{Automatic Control Laboratory, D-ITET, ETH Z{\"u}rich, ETL K12, Physikstrasse 3, 8092, Z{\"u}rich, Switzerland}
\address[a3]{Laboratory for Information and Decision Systems, Massachusetts Institute of Technology, Cambridge, MA, USA}
\address[a2]{Electrical and Computer Engineering Department, University of British Columbia, Vancouver, Canada}
\begin{abstract} 
Considering the sequential clearing of energy and reserves in Europe, 
enabling inter-area reserve exchange requires optimally allocating inter-area transmission capacities between these two markets. To achieve this, we provide a market-based allocation framework and derive payments with desirable properties. The proposed min-max least core selecting payments achieve individual rationality, budget balance, and approximate incentive compatibility and coalitional stability. The results extend the works on private discrete items to a network of continuous public choices.
\end{abstract}

\end{frontmatter}

\section{Introduction}  

The increasing penetration of stochastic renewable generation poses new challenges to the electricity markets that were conceived on the premise of predictable and fully controllable generation.
In the current European market design, energy and reserve capacity are traded through independent and sequential auctions, which are commonly executed at noon the day before actual operation, based on point forecasts of renewable energy production. Any imbalance between scheduled generation and load demand arising close to the hour of delivery is managed through the balancing market.
Even though many recent works demonstrated the benefits of stochastic market design \citep{pritchard2010single,cory2018payment,gerard2018risk,ordoudis2018energy}, the actual implementation of such approaches would require significant restructuring of any of the existing frameworks. Owing to this reason, we restrict ourselves to the status-quo sequential market architecture.
\looseness=-1

Apart from the limited temporal coordination between scheduling and balancing decisions, the European market suffers also from partial coordination in space. Although day-ahead energy markets are jointly cleared, reserve and balancing markets are still operated on a regional (country) level. 
To mitigate this inefficiency, the European Commission regulation has already published a \textit{reserve exchange} guideline to be completed by 2023~\citep{EC1}. 
However, the joint-clearing of reserve markets requires also allocating a portion of the inter-area transmission capacity from the day-ahead energy market to the reserve market. Currently, these cross-border capacities for the day-ahead market are decided by the operators respecting the requirements in~\citep[Article 16(8a-8b)]{eur19}, see also the example in~\citep[\S 2]{elia20}. A critical issue that remains debatable is the exact methodology for withdrawing a portion of these capacities for reserves and its remuneration~\citep[(14)]{eur19}.

Recent work of \citet{delikaraoglou2018optimal} developed a centralized preemptive transmission allocation model that defines the optimal inter-area transmission capacity allocation for reserves, {using} a stochastic bilevel programming problem that anticipates {the reaction of all} subsequent markets.
This model assumed implicitly full coordination among the regional operators. However, it did not suggest an area-specific benefit allocation which guarantees that these operators and also their areas (which includes producers, load serving entities, and the consumer base) have sufficient benefits to accept the proposed solution. 
This would be concerning since an application for reserve exchange can be filed by even two operators~\citep[(14)]{eur19}. {In coalitional game theory, such arrangements would be called coalitional deviations.} For the simpler setting of imbalance netting (IN), market stakeholders have already recognized that the financial benefits should be shared among the participating operators in a way that every operator and also its area  benefit from the cooperation and have the incentive to continue their participation~\citep{avramiotis2018investigations}. In the same vein, the stakeholder document in \cite[\S 6]{IGCC}, developed by ten European operators, describes a fair benefit allocation for IN. Thus, this will be an actual issue, while we go beyond IN, and implement exchange of balancing services. Motivated by this, our previous work \citep{karaca2019benefits} suggested a benefit allocation for the preemptive model which guarantees that the areas have sufficient benefits to not form subcoalitions. 

A centralized approach involving a preemptive model assumes the availability of private regional information (such as generator bids, demand profiles, and forecast scenarios) and an agreement on the estimates of a considerable number of parameters. The level of transparency still remains as a major concern even in the inter-area transmission capacity calculations for day-ahead markets~\citep{eurel20}. Considering this issue, this paper proposes a market mechanism that requires operators to submit only valuations over transmission capacity allocations, prior to day-ahead and reserve market clearing. This was in fact mentioned as the market-based allocation process with an explicit auction by the regulation in~\citep{EC1}.  We kindly refer to \citep{delikaraoglou2018optimal,EC1} for other schemes being discussed.
\looseness=-1

When designing market mechanisms, four desirable properties are individual rationality, budget balance, incentive compatibility, and coalitional stability~\citep{krishna2009auction}.
A mechanism is individually rational if and only if participants are better off bidding than not. A mechanism is budget balanced if and only if the market organizer does not collect or inject any funds. A mechanism is (dominant strategy) incentive compatible if and only if participants prefer bidding truthfully regardless of the bids of the others.
Finally, a mechanism is coalitionally stable if and only if there is no subset of participants that would rather form their own market.

There are well-known impossibility results on achieving budget balance and incentive compatibility simultaneously for both public and private good problems~\citep{green1979incentives,myerson1983efficient}. 
As a result, the goal is generally to ensure a reasonable trade-off between these properties. 
On the other hand, the existing works for coalitional stability are restricted to auctions of private goods, and
they do not study budget balance. The authors in \citep{day2008core} characterized these mechanisms by selecting payments from the core (i.e., the set of coalitionally stable outcomes) in auctions of private discrete items, whereas \citep{day2007fair} showed that these mechanisms can approximate incentive compatibility. In contrast to these works, transmission capacity allocations constitute a network of continuous public choices, which generalizes the economic concepts of public goods and public bads. Two important remarks have to be made regarding the problem at hand. First, the shares of transmission capacity for reserves are dictated by the agreement of the operators at the both ends of the tie-line, and they change the costs of the operators and their areas. Second, valuations of the operators are not restricted to being continuous, monotone, or positive in their quantity. All these features mean the existing results on public/private goods are not readily applicable. 
\looseness=-1 

 Our contributions are as follows. We put forward a market framework for withdrawing inter-area transmission capacities from day-ahead energy for reserves. The well-known Groves mechanism is shown to be incentive compatible, but inconsistent with budget balance and coalitional stability. Since coalitional stability is shown to be not attainable in general, we characterize the class of budget balanced mechanisms that approximate coalitional stability, and show that they limit group manipulations and they are individually rational. We show that, among this class of mechanisms, the min-max least core selecting mechanism approximates incentive compatibility. These results are not formalized in any previous work and they extend studies on mechanisms that select payments from the core in auctions of private discrete items. Since the core can be empty for a network of continuous public choices, we derive extensions of these works by relaxing the core to the least core~\citep{maschler1979geometric}. 
 \looseness=-1

\section{Market-based approach}

\subsection{Market framework}

   Let $A$ denote the set of areas. Similar to~\citep{kristiansen2018mechanism,avramiotis2018investigations} and many others, in our framework, area as a whole (country or region) is an ensemble of consumers and generators pertaining to that area and operator (and the transmission owners). This definition ensures the institutional relevance of the overall problem. Even when the main actors involved in the decision-making of a reserve exchange are operators, they are expected to seek the benefits of their areas, see \citep{IGCC,avramiotis2018investigations}, and they are often checked by regulatory authorities.
   \looseness=-1
   
   Let $E$ denote the set of links. Assume the graph $(A,E)$ is strongly connected {and} simple, {i.e., undirected graph without self-loops}. Areas $a_1,a_2\in A$ are adjacent (or neighbors) if $e=(a_1,a_2)=(a_2,a_1)\in E$. A link $e\in E$ is incident to an area $a$ if $a\in e$. Given a set of areas $S\subseteq A$, $E_S=\{e\in E\,|\, \exists!a\in S\ \text{such that}\ a\in e\}$ denotes the set of links connecting the area set $S$ to the remaining areas in $A\setminus S$. Throughout the paper, we use $a$ and $\{a\}$ interchangeably, for instance, $E_a=E_{\{a\}}.$ Moreover, given a set of areas $R\subseteq A$, $E^R=\{e\in E\,|\, \exists a_1,a_2\in R\ \text{such that}\ e=(a_1,a_2)\}$ denotes the set of links connecting the areas within the set~$R.$ 
   Let $\chi\in[0,\,1]^E$ denote the percentage of inter-area interconnection capacity withdrawn from the day-ahead market and allocated to reserves exchange. For any $F\subseteq E$, let $\chi_{F}=\{\chi_e\}_{e\in F}$. We assume its default value is $\chi'\in[0,\,1]^E$ which originates from the existing cross-border agreements. Following our previous discussions, this default value will generally be given by $\chi'= 0,$ which means that reserve exchanges are not admissible in the existing sequential electricity market.  One notable exemption is the Skagerrak interconnector between Western Denmark and Norway with $\chi'=0.15.$
    
    Let $\X_a\subseteq[0,\,1]^{E_a}$ denote the feasible allocation choices for area~$a.$ Each area has a private true valuation $v_a:\X_a\rightarrow\R$, which maps from transmission capacity allocations for the links incident to area~$a$ to the change in the cost of area~$a$ relative to the cost under the default values $\chi'$. Here, the cost of an area refers to minus the social welfare, which is given by the sum of the consumers' and generators' surplus pertaining to that area and the congestion rents collected by the corresponding area operator. This definition was established also in \citep{kristiansen2018mechanism} for allocating benefits from new interconnections.
    Positive $v_a$ is a reduction in costs. We further have $\chi'_{E_a}\in\X_a$ and $v_a(\chi'_{E_a})=0$. We assume these values can be estimated by the area operators and is thus reflected in their bids. In the numerics, we will address how they can be computed.
    
    Each area then submits a (potentially nontruthful) bid of the form $b_a:\hat{\X}_a\rightarrow\R$, where $\chi'_{E_a}\in\hat{\X}_a\subseteq[0,\,1]^{E_a}$ and $b_a(\chi'_{E_a})=0.$ Observe that the variable $\chi_{(a,a')} $ is a public choice shared by $a$ and $a'$, since $\chi_{(a,a')}$ is an argument to the bids of both areas $a$ and {$a'$}. These functions can {theoretically} be extended to the links incident to the neighbors. 
    However, an area is assumed to be precluded from taking part in the decision for other links. 
    \looseness=-1
    
    Given the strategy profile $\BC=\{b_a\}_{a\in A},$ a \textit{mechanism} defines an \textit{allocation rule} $\chi^*(\BC)\in[0,\,1]^{E}$ and a \textit{payment rule} $p_a(\BC)\in\R$ for all $a\in A$. We restrict our attention to the allocation rule that achieves the allocative efficiency:
    \begin{equation}\label{eq:market_prob}
        V(\mathcal{B})= \max_{\chi\in\SB} \sum_{a\in A}b_a(\chi_{E_a})\ \ \mathrm{s.t.}\ \ \chi_{E_a}\in\X_a,\ \forall a\in A,
\end{equation}
where the set $\SB\subseteq[0,\,1]^{E}$ (with $\chi'\in\SB$) {may encompass any exogenously imposed regulatory constraints}, e.g., a fixed percentage of reserves should be covered by internal resources, or a restriction on the feasible $\chi$ for computational tractability. Let the optimal solution of \eqref{eq:market_prob} be denoted by $\chi^*(\BC)$. Assume that in case of multiple optima there is a tie-breaking rule.  
\looseness=-1

Utility of area $a$ is assumed to be quasilinear (linear and separable in the payment), and it is defined by $u_a(\mathcal{B})=v_a(\chi^*_{E_a}(\mathcal{B}))-p_a(\mathcal{B}).$ Revealed utilities are defined as the utilities computed from the information disclosed to the market organizer: $\bar{u}_a(\mathcal{B})=b_a(\chi^*_{E_a}(\mathcal{B}))-p_a(\mathcal{B}).$ For the market organizer, both its revealed and true utilities are equivalent, and defined by the total payment collected: $u_{\text{MO}}(\mathcal{B})=\bar{u}_{\text{MO}}(\mathcal{B})=\sum_{a\in A}p_a(\mathcal{B}).$
These definitions are imperative since the true utilities of the areas are
unknown. Notice that the revealed utilities correspond to the true utilities whenever the submitted bids are the true valuations. Hence, any property defined over the revealed utilities would also hold for true utilities whenever the areas are truthful. Moreover, any payment rule is uniquely defined by the revealed utilities it induces:  $p_a(\mathcal{B})=b_a(\chi^*_{E_a}(\mathcal{B}))-\bar{u}_a(\mathcal{B}).$

Before providing the desired fundamental properties our mechanisms, we highlight that the framework studied in this paper defines monetary quantities on an area level, and does not distribute them on a market participant level. Defining such distribution rules is part of our ongoing work, see also the discussions in \cite{karaca2019benefits,karaca2020theory}.
\begin{definition}[Individual rationality]
    A mechanism is individually rational if the areas are facing nonnegative revealed utilities: $\bar{u}_a(\mathcal{B})=b_a(\chi^*_{E_a}(\mathcal{B}))-p_a(\mathcal{B})\geq0.$
    \end{definition}
\begin{definition}[Budget balance]
    A mechanism is budget balanced if the market operator faces a nonnegative (revealed) utility $\bar{u}_{\text{MO}}(\mathcal{B})\geq 0.$ Even more preferably, a mechanism is strongly budget balanced if $\bar{u}_{\text{MO}}(\mathcal{B})=0.$
\end{definition}
\begin{definition}[Efficiency]
    A mechanism is efficient if the sum of all utilities ${u}_{\text{MO}}(\mathcal{B})+\sum_{a\in A}u_a(\mathcal{B})=\sum_{a\in A}v_a(\chi^*_{E_a}(\mathcal{B}))$ is maximized. Or equivalently, efficiency is attained if we are solving for the optimal allocation of the market in \eqref{eq:market_prob} under the condition that all of the areas submitted their true valuations $\VC=\{v_a\}_{a\in A}.$
\end{definition}

Efficiency can also be defined for the revealed utilities. However, this property would not be meaningful since it is guaranteed independent of the payment rule as long as we are solving for the optimal allocation of~\eqref{eq:market_prob} under the submitted bids. Connected with the original efficiency definition and its relation to truthfulness, we bring in incentive compatibility.

\begin{definition}[Incentive compatibility]
    A mechanism is dominant-strategy incentive-compatible (DSIC) if the true valuation profile $\VC=\{v_a\}_{a\in A}$ is the dominant strategy Nash equilibrium.
\end{definition}

Unilateral deviations are not the only
manipulations we need to consider in order to ensure that all
areas reveal their true valuations to the market. A subset of areas $S\subset A$ can potentially exercise a coalitional deviation, that is, they can {exclude} areas $A\setminus S$ to form their own market~\citep[(14)]{eur19} and compute the optimal transmission allocation for only their bids,
\begin{equation}\label{eq:market_prob_coal}
    \begin{split}
        V(\mathcal{B}_S)= \max_{\chi\in\SB} \sum_{a\in S}b_a(\chi_{E_a})\ \ \mathrm{s.t.}\ \ \chi_{E_a}\in\X_a,\ \forall a\in S,\\ \chi_e=\chi_e',\ \forall e\in E\setminus E^S,
    \end{split}
\end{equation}
where $\mathcal{B}_S=\{b_j\}_{j\in S}.$
The last set of constraints encodes the fact that altering the transmission capacity allocation of a link requires the approval from both areas incident to it. This observation implies that a deviating coalition cannot change the default value $\chi'$ for links missing these two approvals. Observe that $V(\mathcal{B}_{a})=0$ for all $a\in A$, since a single area cannot change any of the default values. It can easily be verified that the set function $V$ is nondecreasing in~$S$. 
Related to coalitional deviations, we bring in $\epsilon$-coalitional stability. 

\begin{definition}[$\epsilon$-Coalitional stability]
    Given $\epsilon\ge0$, a mechanism is $\epsilon$-coalitionally stable if the revealed utilities of the areas lie in the $\epsilon$-core $K_\text{Core}(\BC,\epsilon)$, that is,
$\{\bar{u}_a(\mathcal{B})\}_{a\in A}\in K_\text{Core}(\BC,\epsilon)=\{\bar u \in \R^A\,|\,\sum_{a\in A}\bar{u}_a=V(\BC),\ \sum_{a\in S}\bar{u}_a\geq V(\BC_S)-\epsilon,\ \forall S\subset A\}.$
\end{definition}


%

As a remark, the core is defined as $K_\text{Core}(\BC,0)$. Existing results concerning the (non)emptiness of the core are relegated to \ref{app:A}. In our problem instances the core is usually empty. Hence, we utilized the {notion of} $\epsilon$-core from \citep{shapley1966quasi} in the above definition. Invoking the definition of the revealed utilities, the equality constraint in $K_\text{Core}(\BC,\epsilon)$ is equivalent to strong budget balance, since otherwise areas would prefer to arrange this market with another organizer. When $\epsilon=0$, the inequalities are our exact coalitional requirement: no set of areas can improve their total revealed utilities by a coalitional deviation. This property is defined over the revealed utilities, since the true utilities are private information. Whenever $\epsilon>0$, these inequalities can be interpreted as follows. If organizing a coalitional deviation entails an additional utility reduction of $\epsilon\in\mathbb{R}$, the total revealed utilities would be given by $V(\BC_S)-\epsilon$. Then, the resulting core would be the $\epsilon$-core. This provides us with an $\epsilon$ approximation of coalitional stability.

In order to motivate our proposal in the next section, we briefly review the well-known Groves payment defined by
$p_a(\BC)=b_a(\chi^*_{E_a}(\mathcal{B}))-(V(\mathcal{B})-h_a(\mathcal{B}_{-a}))$, where $\mathcal{B}_{-a}=\{b_j\}_{j\in A\setminus a}$. A particular choice for the function $h_a(\mathcal{B}_{-a})\in\R$ is the \textit{Clarke pivot rule} $h_a(\mathcal{B}_{-a})=V(\mathcal{B}_{-a})$
where $V(\mathcal{B}_{-a})$ is the optimal value of \eqref{eq:market_prob_coal} with $S=A\setminus a$. The Groves payment with the Clarke pivot rule is referred to as the Vickrey-Clarke-Groves (VCG) mechanism.
Our first result shows that the properties of the VCG mechanism extend to our problem. This result is a straightforward generalization of the original proof~\citep{krishna2009auction}, which does not consider a network of continuous (divisible) public choices with general nonconvex constraints and nonconvex valuations, and included for the sake of completeness.

\begin{proposition}\label{prop:VCG}
	\label{thm:incentive_comp}
	Given the model \eqref{eq:market_prob}, (i) the Groves payment yields a DSIC mechanism, (ii) the VCG mechanism (the Groves payment with the Clarke pivot rule) ensures \text{individual rationality}.
\end{proposition}

The proof is relegated to \ref{app:B}. In summary, all bidders have incentives to reveal their true valuations under the Groves payment if they consider only unilateral deviations. Moreover, this mechanism is known to be the unique DSIC mechanism for a general class of problems~\citep{green1977characterization}. 
However, two negative results can be stated from the literature as follows. There is no mechanism for public good problems that can always solve for the optimal allocation under the submitted bids, and attain DSIC and strong budget balance simultaneously~\citep{green1979incentives}. 
The work in \citep{myerson1983efficient} ensures the above for the exchange of private goods in the more general setting of Bayesian implementation. 
\ref{app:B2} presents a counterexample for our problem involving a network of public choices showing that Groves payment is not strongly budget balanced, $\epsilon$-coalitional stability is not attained, and it cannot be efficient.

Next section proposes a payment rule that attains strong budget balance, individual rationality, and approximates coalitional stability and DSIC.

\subsection{Our proposal for the market-based approach}

Let $\epsilon^*(\BC)$ be the critical value such that the $\epsilon$-core is nonempty, $\epsilon^*(\BC)=\min\{\epsilon\geq 0\,|\,K_\text{Core}(\BC,\epsilon)\not=\emptyset\}$. The set $K_\text{Core}(\BC,\epsilon^*(\BC))$ is called the \textit{least core}~\citep{maschler1979geometric}. In contrast to \citep{maschler1979geometric}, we additionally include the constraint $\epsilon\geq 0$, since $\epsilon^*(\BC)<0$ certifies that the core is nonempty, and our goal---the exact coalitional stability implied by $K_\text{Core}(\BC,0)$---is achievable.

The least core attains minimal violation of our coalitional requirement whenever the core is empty. This establishes the foundation for the following definition.

\begin{definition}[Approximate coalitional stability]
    A mechanism is approximately coalitionally stable if the revealed utilities lie in the least core $K_{\text{Core}}(\BC,\epsilon^*(\BC))$.
\end{definition}

 Observe that strong budget balance is implied by the least core $K_{\text{Core}}(\BC,\epsilon^*(\BC))$, $\sum_{a\in A}\bar{u}_a(\mathcal{B})=V(\BC)$, but individual rationality is not, since we have $\bar{u}_a(\mathcal{B})\geq V(\BC_a)-\epsilon^*(\BC)=-\epsilon^*(\BC)$.
 
 We now define the least core selecting payment rule as $p_a^{\text{LC}}(\mathcal{B})=b_a(\chi^*_{E_a}(\mathcal{B}))-\bar{u}_a(\mathcal{B}),\ \text{where}\  \bar{u}(\mathcal{B})=\{\bar{u}_a(\mathcal{B})\}_{a\in A}\in K_{\text{Core}}(\BC,\epsilon^*(\BC)).$
This mechanism is strongly budget balanced and approximately coalitionally stable. We can prove two additional properties. First, we can obtain a bound on the profitability of picking nontruthful bids as a group of areas. This result complements approximate coalitional stability, since a coalition can resort to such group manipulations while still being part of the market with all the areas. Second, we prove individual rationality.

\begin{theorem}\label{thm:main1}
Given the model \eqref{eq:market_prob}, 
    \begin{enumerate}
    \item[(i)] Let $\bar{\epsilon}$ be an upper bound on $\epsilon^*(\BC)$ for all admissible profiles $\BC$. Assume a coalition of areas $S\subset A$ is strategizing as a group to pick their bid functions~$\mathcal{B}_{S}$. Under the least core selecting payment rule, they can obtain at most $\bar{\epsilon}$ more total utility when compared to the case in which they participate as a single area in a VCG mechanism.
		\item[(ii)] The least core selecting payment rule yields an individually rational mechanism.
    \end{enumerate}
\end{theorem}

To prove the above result, we first bring in a lemma reformulating the least core with an alternative set of inequalities in the form of upper bounds.
\looseness=-1

\begin{lemma}\label{lem:1}
$\bar{u}(\mathcal{B})\in K_{\text{Core}}(\BC,\epsilon^*(\BC))$ if and only if $\sum_{a\in A}\bar{u}_a(\mathcal{B})=V(\BC)$ and $\sum_{a\in S}\bar{u}_a(\mathcal{B})\leq V(\BC)-V(\BC_{-S})+\epsilon^*(\BC)$, where $\BC_{-S}=\{b_j\}_{j\in A\setminus S}$ for all $S\subset A.$
\end{lemma}
\begin{proof}
Since $\sum_{a\in R}\bar{u}_a(\mathcal{B})=V(\mathcal{B})-\sum_{a\in A\setminus R}\bar{u}_a(\mathcal{B})$, we can equivalently reorganize the inequality constraints as $V(\mathcal{B})-\sum_{a\in A\setminus R}\bar{u}_a(\mathcal{B})\geq V(\BC_{R})-\epsilon^*(\BC)$, for all $R \subset A$.
	Setting $R=A\setminus S$ yields the statement. \hfill
\end{proof}
\begin{proof}[Proof of Theorem~\ref{thm:main1}]
\textit{(i)} For the set of areas $S$, define a merged bid for the case in which they participate as a single area $j$: $b_j:\hat{\X}_j\rightarrow\R,$ where $\hat{\X}_j=\{\chi_{E_S}\,\rvert\,\exists\chi_{E^S}\ \text{such\ that}\ \chi_{E_a}\in\hat{\X}_a,\forall a\in S\},$ and ${b}_j(\chi_{E_S})=\min_{\chi_{E^S}}\, \sum_{a\in S}b_a(\chi_{E_a})\ \ \mathrm{s.t.}\ \ \chi_{E_a}\in\hat{\X}_a,\ \forall a\in S$. Let $E_j=E_S.$ Observe that $b_j(\chi_{E_j}')\geq 0$, and whenever this value is strictly greater than 0, we can normalize it by assuming that the set $S$ collected the resulting positive revealed utility before participating in the market. Using the same definition, denote the true merged valuation by $v_j:{\X}_j\rightarrow\R.$ Define the profiles ${\tilde{\mathcal{B}}}=(\mathcal{B}_{-S},{\mathcal{B}}_{j})$, ${{\mathcal{B}}}=(\mathcal{B}_{-S},{\mathcal{B}}_{S}),$ ${\tilde{\mathcal{V}}}=(\mathcal{B}_{-S},{\mathcal{V}}_{j})$.
Total utility obtained from group bidding by the set $S$ is given by
	\begin{align*}
	\sum_{a\in S} u_a(\BC) &=\sum_{a\in S}[v_a(\chi_{E_a}^*(\BC))-b_a(\chi_{E_a}^*(\BC))+\bar{u}_a(\BC)]\\
	&\hspace{-1.4cm}\le \sum_{a\in S}v_a(\chi_{E_a}^*(\BC))-b_a(\chi_{E_a}^*(\BC))+V(\BC)-V(\BC_{-S})+\epsilon^*(\BC)\\
	&\hspace{-1.4cm}\leq\sum_{a\in S}[v_a(\chi_{E_a}^*(\BC))]-b_j(\chi_{E_j}^*(\BC))+V(\tilde{\BC})-V(\BC_{-S})+\bar{\epsilon}\\&\hspace{-1.4cm}= u_a^{\text{VCG}}(\tilde{\BC})+\bar{\epsilon}\leq u_a^{\text{VCG}}(\tilde{\VC})+\bar{\epsilon}.
	\end{align*}
The first equality is the definition of true utilities. The first inequality follows from the least core selecting payment rule and Lemma~\ref{lem:1}. The second inequality holds since $\epsilon^*(\BC)\le \bar{\epsilon}$ and the definition of the merged bid implies that $V({\tilde{\mathcal{B}}})=V({{\mathcal{B}}})$ and $b_j(\chi_{E_j}^*(\BC))=\sum_{a\in S}b_a(\chi_{E_a}^*(\BC))$. The second equality follows from the definition of the VCG utility and the uniqueness guaranteed by the tie-breaking rule: $\chi^*_{E_j}(\BC)=\chi^*_{E_j}(\tilde{\BC})$. The third inequality is the DSIC property of the VCG mechanism. Therefore, the utility obtained from group manipulation is upper bounded by the utility obtained when the group participates as a single area in a VCG mechanism plus $\bar{\epsilon}.$

\textit{(ii)} For this proof, we need to show that the least core selecting payment rule implies that $\bar{u}_a(\mathcal{B})=b_a(\chi^*_{E_a}(\mathcal{B}))-p_a(\mathcal{B})\geq 0$ under any bid profile for any area. Clearly, this is equivalent to certifying that $K_{\text{Core}}(\BC,\epsilon^*(\BC))\subseteq \R^A_+$ under any bid profile. To this end, we can extend the method in \citep[Theorem 2.7]{maschler1979geometric} by taking into account that our set function~$V$ is defined by the optimization problem in~\eqref{eq:market_prob_coal}, and hence it is both nondecreasing and $V(\mathcal{B}_{a})=0$ under any bid profile. 

Fixing the bid profile to be $\BC$, we now prove by contradiction. Let $u \in K_{\text{Core}}(\BC,\epsilon^*(\BC))$ with $u_{a'}<0$. In this case, we show that there exists $\epsilon <\epsilon^*(\mathcal{B})$ such that $K_\text{Core}(\mathcal{B},\epsilon)\neq \emptyset$, which would contradict the definition of the least core. Since $u \in K_\text{Core}(\mathcal{B},\epsilon^*(\mathcal{B}))$ and $V(\mathcal{B}_{a'})=0$, we have $0>u_{a'}\geq V(\mathcal{B}_{a'}) -\epsilon^*(\mathcal{B})=-\epsilon^*(\mathcal{B})$. For any $S\not\ni a'$, we have $\sum_{a\in S}u_a+u_{a'}\geq \mathcal{V}(\BC_S\cup \BC_{a'})-\epsilon^*(\mathcal{B})$ (use the fact that $\epsilon^*(\mathcal{B})\geq 0$ for the case corresponding to $S\cup a'=A$). As previously mentioned, $V$ is nondecreasing and $u_{a'}<0$. Hence, we obtain $\sum_{a\in S}u_a> V(\mathcal{B}_S)-\epsilon^*(\mathcal{B})$.

We can always find a small positive number $\delta$ such that  $\sum_{a\in S}u_a - |S|\delta> V(\mathcal{B}_S)-\epsilon^*(\mathcal{B})+\delta$ holds for any $S\not\ni a'$. Next, we show that $K_\text{Core}(\mathcal{B},\epsilon^*(\mathcal{B})-\delta)$ is nonempty. Define $\tilde{u}$ such that $\tilde{u}_a=u_a - \delta$ for all $a\neq a'$and $\tilde{u}_{a'}={u}_{a'}+(|A|-1)\delta$. Revealed utility $\tilde{u}$ clearly satisfies the equality constraint in $K_\text{Core}(\mathcal{B},\epsilon^*(\mathcal{B})-\delta)$. For inequality constraints $S\not\ni a'$, we have $\sum_{a\in S}\tilde{u}_a=\sum_{a\in S}u_a - |S|\delta> V(\mathcal{B}_S)-\epsilon^*(\mathcal{B})+\delta$, where the strict inequality follows from the definition of $\delta$. For inequality constraints $S\ni a'$, we have $\sum_{a\in S} \tilde{u}_a\geq\sum_{a\in S}u_a+\delta\geq V(\mathcal{B}_S)-\epsilon^*(\mathcal{B})+\delta$. Hence, $\tilde u \in K_\text{Core}(\mathcal{B},\epsilon^*(\mathcal{B})-\delta)$, and  $K_\text{Core}(\mathcal{B},\epsilon^*(\mathcal{B})-\delta)\neq\emptyset.$ This observation can be done under any bid profile and hence it concludes the proof. \hfill
\end{proof}

Whenever $\bar{\epsilon}=0,$ the core is nonempty, part \textit{(i)} achieves the utility bounds derived for bidding with multiple identities (shill bidding) in core-selecting auctions of private discrete items~\cite[Theorem 1]{day2008core}.\footnote{We kindly refer to~\cite{karaca2019designing,karaca2018core,karaca2020theory} for the applications of core-selecting auctions in an electricity market setting.} In contrast, our problem involves a network of continuous public choices. There are two major differences this entails in terms of proof method. First, the core in auctions of private discrete items is always nonempty, and it involves also the utility of the market organizer ignoring the budget balance property. Similar conclusions cannot be made for our problem. Our proof method utilizes instead the least core constraints without the market organizer and also an upper-bound on its relaxation term~$\epsilon$. Second, having a network of public choices requires integrating a novel definition of how areas can merge and participate as a single area into the proof method, whereas this is not needed in \citep{day2008core}. Finally, we highlight that the part \textit{(ii)} of our theorem is not implied directly by the definition of the least core (as it was the case in the core: $\bar{u}_a(\mathcal{B})\geq V(\BC_a)=0$) since individual rationality constraints are relaxed by the nonnegative term $\epsilon^*(\BC)$. 

Because we are deviating from using the DSIC Groves payment, we also need to quantify the violation of this property.

\begin{theorem}\label{thm:unil}
    Given the model \eqref{eq:market_prob}, let $p$ be any payment rule that charges at most $\bar{\epsilon}$ (with $\bar{\epsilon}\geq 0$) less than the VCG mechanism under the same bid profile. The additional utility of an area $a$ by a unilateral deviation from its true valuation, that is, $u_a(\BC_a\cup \BC_{-a})-u_a({\VC_a\cup \BC_{-a}})$ for a nontruthful bid~$\BC_a$, is at most $\bar{\epsilon}+u^{\text{VCG}}_a(\VC_a\cup \BC_{-a})-u_a({\VC_a\cup \BC_{-a}})$, where $u^{\text{VCG}}_a(\VC_a\cup \BC_{-a})=V(\VC_a\cup \BC_{-a})-V(\BC_{-a})$ is the VCG utility under the true valuation. 
\end{theorem}
\begin{proof}
    	Assume there exists a bid $\hat b_a:\hat \X_a \rightarrow \R$ such that 
	\begin{equation*}
	\begin{split}
	 &[v_a(\chi^*_{E_a}(\hat{\BC}_a\cup \BC_{-a}))-p_a(\hat{\BC}_a\cup \BC_{-a}) ]-u_a({\VC_a\cup \BC_{-a}})\\ &\hspace{4cm}> \bar{\epsilon}+V(\VC_a\cup \BC_{-a})-V(\BC_{-a})-u_a({\VC_a\cup \BC_{-a}}),
	 \end{split}
	\end{equation*}
	where $\hat{\BC}_a=\{\hat b_a\}$, $\chi^*_{E_a}(\hat{\BC}_a\cup \BC_{-a})$ is the optimal allocation of the problem corresponding to $V(\hat{\BC}_a\cup \BC_{-a})$.
	The inequality above is equivalent to the existence of a deviation that is more profitable than the given upper bound. 
	Notice that the following holds by our assumption on $p$, $$p_a(\hat{\BC}_a\cup \BC_{-a})\geq b_a(\chi^*_{E_a}(\hat{\BC}_a\cup \BC_{-a}))+V(\BC_{-a})-V(\hat{\BC}_a\cup \BC_{-a})-\bar{\epsilon}.$$
	Combining the inequalities above, we have
\begin{equation*}
\begin{split}
&v_a(\chi^*_{E_a}(\hat{\BC}_a\cup \BC_{-a}))-\big[b_a(\chi^*_{E_a}(\hat{\BC}_a\cup \BC_{-a}))+V(\BC_{-a})\\ &\hspace{5cm}-V(\hat{\BC}_a\cup \BC_{-a})\big]> V(\VC_a\cup \BC_{-a})-V(\BC_{-a}).
\end{split}
	\end{equation*}Observe that the first term is the VCG utility under a non-truthful bid, whereas the second term is the VCG utility under the true valuation. The strict inequality above contradicts the DSIC property of the VCG mechanism. We conclude that $\bar{\epsilon}+u^{\text{VCG}}_a(\VC_a\cup \BC_{-a})-u_a({\VC_a\cup \BC_{-a}})$ is indeed an upper bound on the additional profit obtained from a unilateral deviation. \hfill
\end{proof}

Notice that the least core-selecting payment rule satisfies the assumption in Theorem~\ref{thm:unil}. Lemma~\ref{lem:1} implies $p_a^{\text{LC}}(\mathcal{B})=b_a(\chi^*_{E_a}(\mathcal{B}))-\bar{u}_a(\mathcal{B})\geq b_a(\chi^*_{E_a}(\mathcal{B}))+V(\BC_{-a})-V(\BC)-\epsilon^*(\BC)=p_a^{\text{VCG}}(\mathcal{B})-\epsilon^*(\BC).$ Letting $\bar{\epsilon}$ be an upper bound on $\epsilon^*(\BC)$ for all admissible bid profiles yields the assumption. In the numerics, this parameter $\bar{\epsilon}$ will be estimated. Whenever $\bar{\epsilon}=0$, we achieve the unilateral deviation bounds originally derived for core selecting auctions of private discrete items in~\cite[Theorem 3.2]{day2007fair}. Theorem generalizes this result to the least core, and also to a network of continuous public choices. The differences of our proof method are the integration of the public choice bidding language and the market function~$V$ and also the generalization to payment mechanisms that can charge less than the VCG mechanism.

We now propose a method to pick a least core selecting payment rule to approximate the DSIC property. As we discussed, any payment rule is uniquely defined by the revealed utilities. First solve the following optimization problem to compute $\epsilon^*(\BC)$ of the least core: \begin{equation}\label{eq:ba_opt_ini}\min\{ \epsilon\,\rvert\,\epsilon\geq 0,\exists\bar{u}\in  K_{\text{Core}}( \BC,\epsilon)\}.\end{equation}
We can then solve the following to obtain the revealed utilities of our payment rule approximating DSIC: 
\begin{equation}
\label{eq:ba_opt_tie}
\min_{\bar{u}}\,\left\{ \max_{a\in A}\ \bar{u}_a-\bar{u}_a^{\text{VCG}}(\mathcal{B})\,\Big\rvert \  \bar{u}\in K_{\text{Core}}( \BC,\epsilon^*(\BC))\right\},
\end{equation}
where $\bar{u}_a^{\text{VCG}}(\mathcal{B})=V({\BC})-V(\BC_{-a})$ is the VCG revealed utilities.
Let $\bar{u}^{\text{MLC}}(\BC)$ denote its optimal solution. 
The min-max least core selecting (MLC) payment rule is defined by $p_a^{\text{MLC}}(\mathcal{B})=b_a(\chi^*_{E_a}(\mathcal{B}))-\bar{u}_a^{\text{MLC}}(\mathcal{B})$ for all $a\in A.$ We refer to it as the MLC mechanism.

\begin{corollary}
    The MLC mechanism is approximately DSIC in the sense that the maximum of the bounds in Theorem~\ref{thm:unil} for all areas is minimal among all the least core selecting payment rules under the condition that all the areas are truthful. 
\end{corollary}

This statement can easily be verified by evaluating the bound in Theorem~\ref{thm:unil} when all the remaining areas are truthful, that is, submitting ${\VC}_{-a}.$  

As a remark, \eqref{eq:ba_opt_ini} and \eqref{eq:ba_opt_tie} can be cast as linear programs since $K_{\text{Core}}( \BC,\epsilon)$ is given by a set of linear equality and inequality constraints. However, a direct solution requires solving \eqref{eq:market_prob_coal} under all sets of coalitions to evaluate the set function $V$. Instead, these two linear programs can also be tackled efficiently using an iterative constraint generation algorithm. At every iteration, the method generates the constraint with the largest violation for a provisional solution. In practice, it converges after a few iterations. This method is well studied for the core of coalitional games, and it can always be implemented if the function~$V$ is defined by an optimization problem. A formulation of this algorithm for the least core can be found in our {previous} work~\citep{karaca2019benefits}. Finally, in this paper,~$V$ is defined by the optimization in~\eqref{eq:market_prob_coal}.
This would allow us to implement constraint generation.
\looseness=-1

\section{An illustrative case study}

We consider Figure~\ref{fig:three_area}, which comprises three areas, to compare the effectiveness of the proposed mechanisms. Following the prevailing approach, assume $\chi'=0.$ 
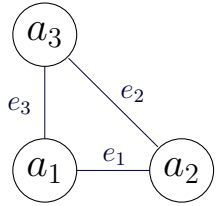
\begin{figure}[h]
	\centering
	\begin{tikzpicture}[scale=1.2, every node/.style={scale=0.9}]
    \draw[-,black!80!blue,line width=.1mm] (0.35,0) -- (1.15,0) node[above=.2cm, left=.2cm] {\large $e_1$};
    \draw[-,black!80!blue,line width=.1mm] (0,0.35) -- (0,1.15) node[above=-.4cm, left=-1.6cm] {\large $e_2$};
    \draw[-,black!80!blue,line width=.1mm] (1.24,0.26) -- (0.26,1.24) node[above=-.7cm, left=.4cm] {\large $e_3$};
    \draw (0,0) circle (.35cm) node {\LARGE $a_1$};
    \draw (1.5,0) circle (.35cm) node {\LARGE $a_2$};
    \draw (0,1.5) circle (.35cm) node {\LARGE $a_3$};
	\end{tikzpicture}
	\caption{Three-area graph}\label{fig:three_area}
\end{figure}

To verify the analysis of this paper, we will focus on the true valuations of areas given by:
\begin{equation*}
    \begin{split}
        v_{a_1}(\chi_{e_1},\chi_{e_3}) &=C_{a_1}\,[\chi_{e_1}-\chi_{e_1}^2,\,\chi_{e_3}-\chi_{e_3}^2,\,\chi_{e_1}\chi_{e_3}]^\top,\\
        \text{with}\ C_{a_1}&=[1.888,\,1.120,\,-5.533],\\
        v_{a_2}(\chi_{e_1},\chi_{e_2}) &=C_{a_2}\,[\chi_{e_1}-\chi_{e_1}^2,\,\chi_{e_2}-\chi_{e_2}^2,\,\chi_{e_1}\chi_{e_2}]^\top,\\
        \text{with}\ C_{a_2}&=[2.262,\,1.710,\,-5.012],\\
        v_{a_3}(\chi_{e_2},\chi_{e_3}) &=C_{a_3}\,[\chi_{e_2}-\chi_{e_2}^2,\,\chi_{e_3}-\chi_{e_3}^2,\,\chi_{e_2}\chi_{e_3}]^\top,\\
        \text{with}\ C_{a_3}&=[1.448,\,2.305,\,-6.580],,
    \end{split}
\end{equation*} where the permitted values of $\chi$ are chosen from $\{0,0.1,0.2,0.3,0.4\}$, and $v_a(0)=0,$ for all $a$.
In practice, each area can estimate its costs from the sequential market at each $\chi$ portion to obtain its valuation. For instance, in~\citep{kristiansen2018mechanism}, each area is assigned its producer and consumer surpluses, and the congestion rent is assumed to be divided equally between the adjacent areas; based on the zonal/nodal prices. A zonal cost estimation using such a cost allocation scheme can initially require an estimation of some market parameters. Given that these auctions are executed repeatedly, the participants can adjust their offers/bids over time via online learning algorithms, see~\citep{karaca2019no}. We expect this to ensure they are adequately remunerated. These aspects are beyond the scope of the present paper, and the bidding languages and bidding algorithms will be addressed in our future work.

Under true valuations, the optimal solution is  $[\chi_{e_1}^*,\chi_{e_2}^*,\chi_{e_3}^*]=[0.4,0,0.2]$. Payments/utilities are provided in Table~\ref{tab:payutup}. In practice, truthfulness might not hold. The VCG mechanism has a total surplus of $\bar{u}_{\text{MO}}(\mathcal{V})=\sum_{a\in A}p_a(\mathcal{V})=0.373.$ In some other cases, the VCG mechanism instead ends up with a deficit. The MLC mechanism is strongly budget balanced, that is, $\bar{u}_{\text{MO}}(\mathcal{V})=\sum_{a\in A}p_a(\mathcal{V})=0.$ The core is empty: $\epsilon^*(\mathcal{V})=\min\{\epsilon\geq 0\,|\,K_\text{Core}(\mathcal{V},\epsilon)\not=\emptyset\}=0.124.$ Under different bid profiles $\epsilon^*$ may vary. We randomize the valuations by picking $10^6$ samples from $C_{a,1}\sim\mathcal{U}([0,3])$, $C_{a,2}\sim\mathcal{U}([0,3])$, $C_{a,3}\sim\mathcal{U}([-9,0])$. Using this, we can estimate $\bar{\epsilon}=0.159$ for Theorems~\ref{thm:main1} and~\ref{thm:unil}. 

\begin{table}[h]\footnotesize 
\begin{minipage}{1\linewidth} 
\centering
\caption{Payments/utilities (\euro)}
\resizebox{.6\textwidth}{!}{\begin{tabular}{|c|c|c|c|c|c|c|}
\hline
 & $p_{a_1}$ & $u_{a_1}$ & $p_{a_2}$ & $u_{a_2}$ & $p_{a_3}$ & $u_{a_3}$\\ \hline
 VCG & $-0.154$& $0.343$ & $0.264$ & $0.279$& $0.263$ & $0.105$\\ \hline
 MLC & $-0.278$& $0.468$ & $0.139$ & $0.404$& $0.139$ & $0.230$\\ \hline
\end{tabular}}\label{tab:payutup}
\end{minipage}
\end{table}
Under the VCG, areas 1 and 2 would strongly prefer to form a coalition, since $V(\mathcal{V}_{a_1}\cup\mathcal{V}_{a_2})=0.996\geq 0.622 = {u}_{a_1}^{\text{VCG}}+{u}_{a_2}^{\text{VCG}}.$ This way, they can increase their total utility by $0.374$.
We briefly illustrate one group manipulation for areas 1 and 2.  Suppose their bids are given by 5 times their true valuations. Under the VCG, their total utility increases from $0.622$ to $1.679$. Under the MLC, their total utility increases from $0.872$ to $0.996$, which is a smaller change than the one for the VCG since Theorem~\ref{thm:unil}-\textit{(i)} limits group manipulation.

In order to enable reserve exchanges between operators, this paper proposed a market framework and derived a payment rule that attains strong  budget  balance,  individual  rationality,  and approximates coalitional stability and DSIC. These results extended studies on mechanisms that select payments from the core in auctions of private discrete items by accounting for the fact that the core can be empty for a network of continuous public choices. Our future work will study rules to distribute these payments on a market participant level. 

\bibliography{library}

\appendix

\section{(Non)emptiness of the core}\label{app:A}

The core is a closed polytope involving $2^{|A|}$ linear constraints. 
    It is nonempty if and only if the function $V$ satisfies the balancedness condition. Balanced problem settings include the cases in which $V$ is supermodular and the cases in which \eqref{eq:market_prob_coal} can be modeled by a concave exchange economy~\citep{shapley1969market}, or a linear production game~\citep{owen1975core}. In their most general form, these results involve an optimization problem maximizing a concave objective subject to linear constraints. We are solving the general non-convex optimization problem \eqref{eq:market_prob}. As a result, these previous works are not applicable to our setup. As is shown in the following proposition, nonemptiness can be guaranteed for a star graph~$(A, E)$. A similar derivation was included in our previous work for a networked coalitional game~\citep{karaca2019benefits}. We prove this result for the problem at hand for the sake of completeness.

\begin{proposition}\label{prop:appcore}

$K_{\text{Core}}(\BC,0)\neq\emptyset$ if $(A,E)$ is a star.
\end{proposition}
\begin{proof}
Let $a\in A$ be the central area. We show that the vector $\bar{u}_a=V(\BC)$, $\bar{u}_j=0$
otherwise lies in the core. Clearly the equality constraint is satisfied.
Observe that star graph implies $V(\BC_{S})=0$ for all $S\not\ni a,$ thus $\sum_{j\in S}\bar{u}_j\geq V(\BC_{S})$ for all $S\not\ni a.$ On the other hand, for all $S\ni a$, $\sum_{j\in S}\bar{u}_j=\bar{u}_a=V(\BC) \geq V(\BC_{S})$ via monotonicity of $V$. \hfill
\end{proof}

\section{Proof of Proposition~\ref{prop:VCG}}\label{app:B}

	(i) For a generic profile $\mathcal{B}$, the utility of area $a$ is:
$u_a(\mathcal{B})=\Big[\sum_{j\in A\setminus a}b_j(\chi^*_{E_j}(\mathcal{B}))+v_a(\chi^*_{E_a}(\mathcal{B}))\Big]-h_a(\mathcal{B}_{-a}),$
	where the term in brackets is the objective of the optimization problem for $V(\VC_a\cup\BC_{-a})$ evaluated at a feasible solution $\chi^*_{E_a}(\mathcal{B}).$ Hence, $u_a(\mathcal{B})\leq V(\VC_a\cup\BC_{-a})-h_a(\mathcal{B}_{-a})$.
    Notice that the term on the right is $u_a(\VC_a\cup\BC_{-a}).$ Therefore, the utility under bidding truthfully weakly dominates the utility under any other bid, regardless of other areas $\mathcal{B}_{-a}$. 
    
    (ii) We have $\bar{u}_a(\mathcal{B})=V(\mathcal{B})-V(\mathcal{B}_{-a}) \geq 0$ by monotonicity of $V$.\hfill$\square$

	\section{The limitations of the Groves payment}\label{app:B2}

    Suppose we have the setting of two areas $a_1$ and $a_2$ connected with a single link. Define $\SB=\hat{\X}_{a_1}=\hat{\X}_{a_2}=\{0,\,1\}$, and $\chi'=0$. Area $a_1$ has two strategies $b_{a_1}^1(1)=1$ and $b_{a_1}^2(1)=0$, area $a_2$ has two strategies $b_{a_2}^1(1)=-1$ and $b_{a_2}^2(1)=0$. Note that these functions are required to be zero at the default transmission capacity allocation $\chi'=0$. Strong budget balance implies the following equalities: $t_{a_1}(\{b_{a_1}^i,b_{a_2}^j\})+t_{a_2}(\{b_{a_1}^i,b_{a_2}^j\})=0,$ for $i,j=1,2.$ Invoking the definition of the Groves payment and \eqref{eq:market_prob}, these equalities can be rewritten as
    \begin{align*}
        &h_{a_1}(\{b_{a_2}^2\})+h_{a_2}(\{b_{a_1}^1\})=-b_{a_1}(\chi^*_{E_{a_1}}(\{b_{a_1}^1,b_{a_2}^2\}))\\&+V(\{b_{a_1}^1,b_{a_2}^2\})
        -b_{a_2}(\chi^*_{E_{a_2}}(\{b_{a_1}^1,b_{a_2}^2\}))+V(\{b_{a_1}^1,b_{a_2}^2\})\\
        &\hspace{3.67cm}=-1+1-0+1=1,\\ &h_{a_1}(\{b_{a_2}^j\})+h_{a_2}(\{b_{a_1}^i\})=0,\ (i,j)\in\{(1,1),(2,1),(2,2)\}.
    \end{align*}
    Above can be written as $A\bm {h}=b=[1,\, 0,\, 0,\, 0]^\top$, where $\bm {h}\in\R^4$ concatenates $h_a$'s. $\text{rank}(A)=3$, moreover $b\notin\text{colspan}(A)$. Hence, there are no functions $h_{a_1}$ and $h_{a_2}$ such that strong budget balance is achieved, and the Groves payment can also not attain $\epsilon$-coalitional stability for any $\epsilon$ (see also the discussions in~\cite{karaca2018exploiting} for the case with a nonempty core).
\end{document}